\documentclass[12pt]{article}

\usepackage[T1]{fontenc}
\usepackage[left=25mm, right=25mm, top=25mm, bottom=25mm]{geometry}
\usepackage{amssymb, amsmath, amscd, amsthm, verbatim, 
	mathrsfs, tikz, color , amsfonts}
\usepackage{cite}
\usepackage{setspace} 
\allowdisplaybreaks
\usepackage[vcentermath, enableskew, noautoscale]{youngtab}
\usepackage{young}

\usepackage{hyperref}

\usepackage{bm}
\usepackage{enumerate}
\usepackage{graphicx}
\usepackage{mathptmx}
\usepackage{enumitem} 
\usepackage{marvosym}    % 提供 \Letter 小信箱图标

\newtheorem{remark}{Remark}[section]
\newtheorem{theorem}{Therorem}[section]
\newtheorem{definition}{Definition}[section]
\newtheorem{lemma}{Lemma}[section]

\title{\textbf{Limit Bi-Shadowing for Semi-Hyperbolic Systems}}
\author{
    Linying Cheng$^{\href{mailto:202306021013@stu.cqu.edu.cn}{\textnormal{\Letter}}}$
    \thanks{Email: {202306021013@stu.cqu.edu.cn}}, 
    Haiye Guo$^{\href{mailto:202206021041t@stu.cqu.edu.cn}{\textrm{\Letter}}}$
 \thanks{Corresponding author, Email:202206021041t@stu.cqu.edu.cn}
    \\
    \small \textit{College of Mathematics and Statistics, Chongqing University, Chongqing 401331, China}
}

\date{}
\begin{document}
	\bibliographystyle{plain}
	\maketitle
	\renewcommand{\abstractname}{}
	\begin{abstract}
		\noindent\textbf{Abstract:}	This paper investigates the shadowing properties in semi-hyperbolic systems.  We introduce three classes of shadowing properties defined on families of manifolds, and prove that a semi-hyperbolic family possesses the $L^p$ bi-shadowing property, the limit bi-shadowing property, and the asymptotic bi-shadowing property under certain conditions. The proof strategy is to transform the shadowing problem into a fixed-point problem, and then apply the Brouwer fixed-point theorem to complete the verification.\\
		\noindent\textbf{Keywords:} Semi-hyperbolic system; $L^p$-bi-shadowing property; Limit bi-shadowing; Asymptotic bi-shadowing; Fixed point theorem
	\end{abstract}

\section{Introduction}

Dynamical systems are an important branch of mathematics, studying the laws of system state evolution over time, and are widely applied in physics, engineering, economics, biology, and other fields. In the early 1960s, structural stability became a significant research topic in dynamical systems. The discovery of systems such as the Smale horseshoe and Anosov torus automorphism revealed that complex chaotic systems possess structural stability, and the analytical condition ensuring their stability is called hyperbolicity. Specifically, if the tangent bundle of a diffeomorphism \(f\) over a compact invariant set \(\Lambda\) can be decomposed into a direct sum of an expanding subbundle \(E^u\) and a contracting subbundle \(E^s\), such that \(Df\) expands on \(E^u\) and contracts on \(E^s\), with expansion and contraction constants independent of the base point, then \(f\) is said to be uniformly hyperbolic on \(\Lambda\). When \(\Lambda\) is the entire space, \(f\) is called an Anosov diffeomorphism. Research results on structural stability can be found in references \cite{stability1, stability3}.

With the deepening of research on structural stability, the concept of pseudo-orbit tracing property emerged. If for any pseudo-orbit of the system, there exists a true orbit that approximates it synchronously, the system is said to have the tracing property. The tracing property not only provides tools for stability research but also offers theoretical guarantees for the reliability of numerical simulations; related research can be found in reference \cite{orbit1}. In 1975, Bowen in \cite{Bowen1} proposed the "shadowing lemma" proving that hyperbolic systems have the shadowing property, marking the beginning of shadowing theory. Subsequently, scholars such as Robinson, Conley, and Shub developed two main proof methods: Anosov's  analytical method (transforming the shadowing problem into a fixed point problem on Banach spaces) and Bowen's geometric method (utilizing the structure of stable and unstable manifolds)(\cite{Robinson, Conley, Shub, Anosov, Bowen1} ).

To address the difficulties of applying hyperbolicity in non-smooth, non-invertible, and non-autonomous systems, Diamond et al. in \cite{semi1} proposed the theory of "semi-hyperbolic systems" in 1995, introducing parameters \((\lambda_s, \lambda_u, \mu_s, \mu_u)\) to characterize the contraction in the stable direction, expansion in the unstable direction, and the coupling strength between them, no longer requiring strict invariance of the tangent bundle decomposition. Moreover, semi-hyperbolic systems still maintain the shadowing property. In 1995, Diamond  et al. in \cite{semi3} proved that semi-hyperbolic systems of Lipschitz mappings possess the bi-shadowing property (combining pseudo-orbit tracing and inverse tracing). The bi-shadowing property not only ensures that pseudo-orbits of the system can be traced by true orbits of a perturbed system but also guarantees that true orbits of the original system can be approximated by true orbits of the perturbed system, thus making the characterization of system tracing behavior more precise. Research on the bi-shadowing property can be found in references \cite{semi3, semi4-bi}.

With the development of theory and applications, researchers have also proposed concepts such as asymptotic pseudo-orbits and average pseudo-orbits, and developed various forms of shadowing, including Lipschitz shadowing, limit shadowing, and inverse shadowing. Among these, the limit shadowing property characterizes the dynamical behavior of a system over infinite time. A dynamical system has the limit shadowing property, meaning that when the single-step error tends to zero as time tends to infinity, the orbit obtained from numerical computation must converge to a true orbit. In 1996, Pilyugin proved that hyperbolic sets have the limit shadowing property in \cite{Lipschitz}. The core idea of the proof is to transform the shadowing problem into a fixed point problem, mainly using the contraction property brought by hyperbolic systems to solve it.

This paper investigates the shadowing properties of semi-hyperbolic families. Based on the existing concept of limit shadowing, we define several new types of limit bi-shadowing properties ($L^p$-bi-shadowing, limit bi-shadowing, and asymptotic bi-shadowing). These properties not only consider the system's tracing behavior on a long time scale but also ensure the system's stability after perturbation using the characteristics of bi-shadowing. Under semi-hyperbolic conditions, contraction is restricted, and the original method used by Pilyugin to prove limit shadowing for hyperbolic sets is no longer applicable. Therefore, this paper provides a new proof method, mainly using the Brouwer Fixed Point Theorem, proving separately for finite and infinite orbits, thereby obtaining the results that semi-hyperbolic families possess $L^p$-bi-shadowing, limit bi-shadowing, and asymptotic bi-shadowing properties.

\section{Main results}
\subsection{Notations}

This paper conducts research on non-stationary systems. The relevant notation is provided below. 

Let $\{(M_i, \langle\cdot, \cdot\rangle)\}_{i\in \mathbb{Z}}$ be a family of compact smooth manifolds, where $\langle\cdot, \cdot\rangle_i$ is the Riemannian metric on $M_i$. Denote by $|\cdot|_i$ and $d_i(\cdot, \cdot)$ the norm induced by $\langle\cdot, \cdot\rangle_i$ on the tangent bundle $TM_i$ and the distance, respectively. For simplicity, we assume that the diameter of each $M_i$ is uniformly bounded by 2. Furthermore, we assume that the sequence $\{(M_i, \langle\cdot, \cdot\rangle)\}_{i\in \mathbb{Z}}$ has uniformly bounded curvature and a uniformly positive injectivity radius. Specifically, 
\begin{equation*}
    \sup_{i\in\mathbb{Z}}\|R_i\|_\infty<+\infty\, \, \, \, \, \, \, \, \text{and}\, \, \, \, \, \, \, \, \inf_{i\in\mathbb{Z}}\text{inj}(M_i)>0, 
\end{equation*}
where $\|\cdot\|_\infty$ denotes the pointwise supremum of the curvature tensor norm, and $\operatorname{inj}(M_i)=\inf\limits_{x_i\in M_i}\operatorname{inj}(M_i, x)$ represents the global injectivity radius of $M_i$.

Take the disjoint union
\begin{equation*}
    \mathbf{M}=\bigsqcup_{i\in\mathbb{Z}}M_i=\bigcup_{i\in\mathbb{Z}}M_i \times i, 
\end{equation*}
we refer to  $\mathbf{M}$ as the total space of $\{M_i\}_{i\in\mathbb{Z}}$, with $M_i$ as a component of $\mathbf{M}$. An element $\bm{x}:=\{x_i\}_i$ on $\mathbf{M}$ with $x_i\in M_i$ for $i\in\mathbb{Z}$. Define a metric on the total space by

\begin{equation*}
    d(x, y)=
    \begin{cases}
        d_i(x_i, y_j), &\text{if }  x_i\in M_i, y_j\in M_j  \, \text{ and } \, i=j;\\
        1, &\text{if }  x_i\in M_i, y_j\in M_j \, \text{ and }\, i\neq j.
    \end{cases}
\end{equation*}

Thus, $\mathbf{M}$ has a natural Riemannian metric $\langle\cdot, \cdot\rangle$ induced by $\{\langle\cdot, \cdot\rangle_i\}_{i\in\mathbb{Z}}$ and the corresponding induced norm $|\cdot|$, i.e., $\langle\cdot, \cdot\rangle|_{M_i}=\langle\cdot, \cdot\rangle_i$ and $|\cdot||_{M_i}=|\cdot|_i(i\in\mathbb{Z})$. 

The transition maps $\bm{f}:\mathbf{M}\to\mathbf{M}$, where for each $i\in\mathbb{Z}$, $\bm{f}|_{M_i}=f_i:M_i\to M_{i+1}$, is defined as
\begin{equation*}
    f^n_i:=
    \begin{cases}
        f_{i+n-1}\circ\cdots\circ f_i:M_i\to M_{i+n}, &n>0;\\
        f_{i-n}^{-1}\circ\cdots\circ f_{i-1}^{-1}:M_i\to M_{i-n}, &n<0;\\
        I_i:M_i\to M_i, &n=0.
    \end{cases}
\end{equation*}
where $I_i$ is the identity map on $M_i$.

Let $\bm{f}:\mathbf{M}\to\mathbf{M}$ be a $C^1$ non-stationary dynamical system. $D\bm{f}$ is said to be \textit{uniformly equicontinuous} if for any $\eta>0$, there exists $\delta>0$ such that for any $i\in\mathbb{Z}$ and $x_i, y_i\in M_i$ with $d(x_i, y_i)<\delta$
\begin{equation*}
    \|Df_i(x_i)-Df_i(y_i)\|<\eta.
\end{equation*}

Now, we present the definition of a semi-hyperbolic family.
\begin{definition}
    A semi-hyperbolic family$\bm{f}$ on $\mathbf{M}$ is satisfing
\begin{enumerate}[start=1, label=(\arabic*), leftmargin=4em]
    \item The tangent bundle $T\mathbf{M}$ has a splitting $T\mathbf{M}=E^s\oplus E^u$ and $D\bm{f}$ is represented by
    \begin{equation*}
        Df_i=
        \begin{pmatrix}
            B_{11i} & B_{12i}\\
            B_{21i} & B_{22i}
        \end{pmatrix}, i\in \mathbb{Z;}
    \end{equation*}
    \item For $i\in\mathbb{Z}$, there exist positive numbers $\lambda_{si}, \lambda_{ui}, \mu_{si}, \mu_{ui}$ such that
    \begin{equation*}
        0<\inf_i\{\lambda^{-1}_{si}\}<1<\inf_i\{\lambda_{\mu i}\}, (1-\lambda_{si})(\lambda_{ui}-1)>\mu_{si}\mu_{ui}
    \end{equation*}
    and
    \begin{equation*}
        m(B_{11i})\ge \lambda_{\mu i}, |B_{22i}|\le \lambda_{si}, |B_{12i}|\le\mu_{ui}, |B_{21i}|\le \mu_{si}. 
    \end{equation*}
        
    \end{enumerate}

\end{definition}

Satisfying the above conditions, $\bm{f}$ is also said to be  $(\lambda_{si}, \lambda_{ui}, \mu_{si}, \mu_{ui})$-semi-hyperbolic family.

For $K\in \mathbf{M}$, if the decomposition $T_K\mathbf{M}=E^s\oplus E^u$ on the tangent bundle at $K$ satisfies an \textit{angle property}, meaning the angle between the subbundles $E^s$ and $E^u$ is uniformly bounded away from zero.

\subsection{Statements of main results}

This paper investigates the limit bi-shadowing property under semi-hyperbolic conditions (collectively referring to $L^p$-bi-shadowing, limit bi-shadowing, and asymptotic bi-shadowing as limit bi-shadowing). The main results are Theorems \ref{Lp-bi}, \ref{limit-bi}, and \ref{Asy-bi}:

\begin{theorem}\label{Lp-bi}
Let $\bm{f}:\mathbf{M}\to\mathbf{M}$ be a semi-hyperbolic family, where $\mathbf{M}=\bigsqcup\limits_{i\in\mathbb{Z}}M_i$. Assume $\bm{f}$ satisfies the angle property and $D\bm{f}$ is uniformly equicontinuous, then $\bm{f}$ has the $L^p$-bi-shadowing property: for any $1\le p\le\infty$, there exist constants $ L>0$ and $ \delta>0$ such that for any sequence $\Delta=\{\delta_i\}_{i\in\mathbb{Z}}$ with  $\|\Delta \|_p\le\delta$, if $\bm{ x}=\{x_i\}_{i\in\mathbb{Z}}$ is an infinite $\Delta$-pseudo-orbit of $\bm{f}$ (i.e., $d(f_i(x_i), x_{i+1})\le\delta_i$ for all $i\in\mathbb{Z}$), and $\bm{g}:\mathbf{M}\to\mathbf{M}$ is a family of continuous maps with
\begin{equation*}
    \|\bm{f}-\bm{g}\|_p:=\left\|\left\{\sup\limits_{{x_i}\in M_i}d(f_i(x_i), g_i(x_i))\right\}_{i\in\mathbb{Z}}\right\|_p\le\delta, 
\end{equation*}
then there exists an infinite orbit $\bm{y}=\{y_i\}_{i\in\mathbb{Z}}$ (i.e., $y_{i+1}=g_i (y_i)$ for all $i\in\mathbb{Z}$), such that\begin{equation*}
\|\{d(x_i, y_i)\}\|_p\leq L\|\Delta\|_p, \, \, \, i\in\mathbb{Z}.
\end{equation*}
\end{theorem}
\begin{remark}
    Theorem \ref{Lp-bi} shows that semi-hyperbolic systems possess the $L^p$-bi-shadowing property. This paper first proves the finite $L^p$-bi-shadowing property for $\bm{f}$, i.e., Theorem \ref{Lp-finite}, and then, building upon the finite case, further obtains the infinite $L^p$-bi-shadowing property, i.e., Theorem \ref{Lp-bi}. The proofs of these two theorems together demonstrate that semi-hyperbolic systems have the $L^p$-bi-shadowing property.
\end{remark}
%\vspace{2pt}
\begin{theorem}\label{limit-bi}
Let $\bm{f}:\mathbf{M}\to\mathbf{M}$ be a semi-hyperbolic family, where $\mathbf{M}=\bigsqcup\limits_{i\in\mathbb{Z}}M_i$. Assume $\bm{f}$ satisfies the angle property and $D\bm{f}$ is uniformly equicontinuous, then $\bm{f}$ has the limit bi-shadowing property: for any sequence $\Delta=\{\delta_i\}_{i\in\mathbb{Z}}$, let $\bm{x}=\{x_i\}_{i\in\mathbb{Z}}\subset \mathbf{M}$ be a $\Delta$-pseudo-orbit of $\bm{f}$ (i.e., $d(f_i(x_i), x_{i+1})\le\delta_i$ for all $i\in\mathbb{Z}$). Let $\bm{g}:\mathbf{M}
      \to\mathbf{M}$ be a family of continuous maps with
      \begin{equation*}
          \lim\limits_{|i|\to\infty} \delta_i = 0 \text{ and } \lim\limits_{|i|\to \infty}\sup_{x_i\in M_i} d(f_i(x_i), g_i(x_i))=0, 
      \end{equation*} then there exists an orbit $\bm{y}=\{y_i\}_{i\in\mathbb{Z}}$ of $\bm{g}$ (i.e., $y_{i+1}=g_i (y_i)$ for all $i\in\mathbb{Z}$), such that
      \begin{equation*}
          \lim\limits_{|i|\to\infty}d(x_i, y_i)=0.
      \end{equation*}
\end{theorem}

%\vspace{2pt}
\begin{theorem}\label{Asy-bi}
    Let $\bm{f}:\mathbf{M}\to\mathbf{M}$ be a semi-hyperbolic family, where $\mathbf{M}=\bigsqcup\limits_{i\in\mathbb{Z}}M_i$. Assume $\bm{f}$ satisfies the angle property and $D\bm{f}$ is uniformly equicontinuous, then $\bm{f}$ has the asymptotic bi-shadowing property: there exists a constant $v_0\in(0, 1)$ such that for any $0<v\le v_0$, if $\bm{x}=\{x_i\}_{i\in\mathbb{Z}}$ is a $v$-asymptotic pseudo-orbit of $\bm{f}$, i.e., there exists a sequence $\Delta=\{\delta_i\}_{i\in\mathbb{Z}}$ with
     \begin{equation*}
         d(f_i(x_i), x_{i+1})\le\delta_i, \forall i\in\mathbb{Z}, 
     \end{equation*}
and
     \begin{equation*}
         \limsup\limits_{|i|\to\infty}\sqrt[|i|]
    \delta_i\le v, 
     \end{equation*}
     and for any family of continuous maps $\bm{g}:\mathbf{M}\to\mathbf{M}$ with
    \begin{equation*}
        \limsup\limits_{|i|\to\infty}\sqrt[|i|]{\sup_{x_i\in M_i} d(f_i(x_i), g(x_i))}\le v, 
    \end{equation*}
    then there exists an orbit $\bm{y}=\{y_i\}_{i\in\mathbb{Z}}$ of $\bm{g}$ (i.e., $y_{i+1}=g_i (y_i)$ for all $i\in\mathbb{Z}$), such that
    \begin{equation*}
        \limsup\limits_{|i|\to\infty}\sqrt[|i|]{d(x_i, y_i)}\le v.
    \end{equation*}
\end{theorem}

\section{\texorpdfstring{$L^p$} --bi-shadowing property}

This section first defines $L^p$-bi-shadowing on the manifold family, then proves the $L^p$-bi-shadowing property for semi-hyperbolic systems. The proof will be presented separately for finite and infinite orbits.

\subsection{Finite \texorpdfstring{$L^p$}--bi-shadowing property}

\begin{definition}\label{Lp-def}

Let $\bm{f}:\mathbf{M}\to \mathbf{M}$ be a map family, where $\mathbf{M}=\bigsqcup\limits_{i\in\mathbb{Z}}M_i$, and $1\le p \le +\infty$. Assume there exist constants $L>0$ and $\delta>0$ such that for any sequence $\Delta=\{\delta_i\}_{i\in\mathbb{Z}}$ with $\|\Delta\|_p\le\delta$, for any $\Delta$-pseudo-orbit $\{x_i\}_{i\in\mathbb{Z}}$ of $\bm{f}$ (i.e., $d(f_i(x_i), x_{i+1})\le\delta_i$ for all $i\in\mathbb{Z}$), and for any continuous map $\bm{g}:\mathbf{M}\to \mathbf{M}$ with
      \begin{equation*}
      \|\bm{f}-\bm{g}\|_p:=\left\|\left\{\sup_{x_i\in M_i}d(f_i(x_i), g_i(x_i))\right\}_{i\in\mathbb{Z}}\right\|_p\le\delta, 
      \end{equation*}
      there exists an orbit $\bm{y}=\{y_i\}$ satisfying $y_{i+1}=g_i(y_i)$ for all $i\in\mathbb{Z}$, and
      \begin{equation*}
          \|\{d(x_i, y_i)\}_{i\in\mathbb{Z}}\|_p\le L\|\Delta\|_p, 
      \end{equation*}
then $\bm{f}$ is said to have the $L^p$-bi-shadowing property.
\end{definition}

\begin{theorem}{(Finite $L^p$-bi-shadowing of semi-hyperbolic systems)}\label{Lp-finite}
Let $\bm{f}:\mathbf{M}\to\mathbf{M}$ be a semi-hyperbolic family, where $\mathbf{M}=\bigsqcup\limits_{i\in\mathbb{Z}}M_i$. Assume $\bm{f}$ satisfies the angle property and $D\bm{f}$ is uniformly equicontinuous, then for any $1\le q\le\infty$, there exist constants $ L>0$ and $ \delta>0, $ such that for any $k\in\mathbb{N}$ and any sequence $\Delta_k=\{\delta_i\}_{i=k}^k, $ with $\|\Delta_k\|_p\le\delta$, if $\bm{ x}=\{x_i\}_{i=-k}^k$ is a finite $\Delta$-pseudo-orbit of $\bm{f}$ (i.e., $d(f_i(x_i), x_{i+1})\le\delta_i$ for  $i=-k, \cdots, k-1$), and $\bm{g}:\mathbf{M}\to\mathbf{M}$ is a family of continuous maps with
\begin{equation*}
    \|\bm{f}-\bm{g}\|_p:=\left\|\left\{\sup\limits_{{x_i}\in M_i}d(f_i(x_i), g_i(x_i))\right\}_{i=-k}^k\right\|_p\le \delta, 
\end{equation*}
then there exists a finite orbit $\bm{y}=\{y_i\}_{i=-k}^k$ satisfying $y_{i+1}=g_i (y_i)$ for  $i=-k, \cdots, k-1$, and\begin{equation*}
\|\{d(x_i, y_i)\}_{i=-k}^k\|_p\leq L\|\Delta\|_p.
\end{equation*}
\end{theorem}
\vspace{1em}

Before proving this theorem, we first state the relevant conclusions and lemmas.

Since the semi-hyperbolic family satisfies the angle property, there exists a constant $h\ge1$ such that
\begin{equation*}
    \|\Pi_{x_i}^s\|, \|\Pi_{x_i}^u\|\le h, i\in\mathbb{Z}, 
\end{equation*}
where $\Pi_{x_i}^s, \Pi_{x_i}^u$ are the \textit{projection operators} $\Pi_{x_i}^s:T_{x_i}M_i\to E_{x_i}^s\, $ and $\Pi_{x_i}^u:T_{x_i}M_i\to E_{x_i}^u, i\in\mathbb{Z}.$

For any $i\in\mathbb{Z}, f_i:M_i\to M_{i+1}$, take $x_i\in M_i$, the exponential map $\exp_{x_i}: T_{x_i}M_i\to M_i$ is a $C^\infty$ diffeomorphism. Define:
\begin{equation*}
    F_i:T_{x_i}M_i(\rho)\to T_{x_{i+1}}M_{i+1}
\end{equation*}
by
\begin{equation*}
   F_i(z_i)=\exp_{x_{i+1}}^{-1}\circ f_i \circ \exp_{x_i}(z_i), i\in\mathbb{Z}, 
\end{equation*}

where $T_{x_i}M_i(\rho)$ denotes the set of vectors in $T_{x_i}M_i$ of length $\le\rho$, i.e., $T_{x_i}M_i(\rho)=\{z\in T_{x_i}M_i:|z|<\rho\}$.

The derivative of $F_i$ at $z_i\in T_{x_i}M_{x_i}(\rho)$ satisfies
\begin{align*}
D(F_i(0_{x_i}))& =D(\exp_{x_{i+1}}^{-1}\circ f \circ \exp_{x_i})(0_{x_i}) = id|_{T_{x_{i+1}}M_{i+1}}\circ Df_i(x_i)\circ id|_{T_{x_i}M_i}=T_{x_i}f_i;\\
 DF_i(z_i) &= D(\exp_{x_{i+1}}^{-1}\circ f \circ \exp_{x_i})(z_i)\\
&=D(\exp_{x_{i+1}}^{-1})(f_i(\exp_{x_i}z_i))\circ Df_i(\exp_{x_i}z_i)\circ D(\exp_{x_i}(z_i)).
\end{align*}
Then $DF_i$ is uniformly equicontinuous on $T_{x_i}M_i(\rho)$.

Take $0<\delta\ll\rho, \Delta=\{\delta_i\}$ with $\|\Delta\|_p\le\delta$, for any $\Delta$-pseudo-orbit of $\bm{f}$ and continuous map $\bm{g}:\mathbf{M}\to\mathbf{M}$ with $\|\bm{f}-\bm{g}\|_p\le\delta.$

Let $G_i(z_i)=\exp_{x_{i+1}}^{-1}\circ g_i \circ \exp_{x_i}(z_i), z_i\in T_{x_i}M_i(\rho), i\in\mathbb{Z}.$

Let positive numbers $L, \delta$ satisfy $L\delta\le \rho$. For each $z_i \in T_{x_i}M_i$ with $|\Pi_{x_i}^sz_i|\leq L\delta$, define the map $F_{z_i}:E_{x_i}^u(L\delta)\to E_{x_{i+1}}^u$ by
\begin{equation*}
    F_{z_i}(\omega_i)=\Pi_{x_{i+1}}^u(F_i(\Pi_{x_i}^sz_i+\omega_i)-F_i(\Pi_{x_i}^sz_i)), 
\end{equation*}
where $E_{x_i}^u(L\delta)=\{z\in E_{x_i}^u:|z|\leq L\delta\}.$

\begin{lemma}\label{F_vi property}
Let $\lambda_u=\inf\limits_{i\in\mathbb{Z}}\{\lambda_{ui}\}>1$. Then for sufficiently small $\delta>0$, there exists a constant $\tilde{\lambda}_u\in (1, \lambda_u)$ such that for any $w_i, w_i'\in E_{x_i}^u(L\delta)$, we have
\begin{equation*}
     |F_{z_i}(w_i)-F_{z_i}(w'_i)|\ge \tilde{\lambda}_u|w_i-w'_i|, i\in\mathbb{Z}.
\end{equation*}
\end{lemma}
\begin{proof}
Take $0<\eta<\frac{\tilde{\lambda}_u-\lambda_u}{h^2}.$

   Since $DF_i(0_{x_i})=T_{x_i}f_i$ and $D\bm{f}$ is uniformly equicontinuous, when the positive numbers $L, \delta$ satisfy $L\delta<\rho$, for any $i\in\mathbb{Z}, x_{i+1}\in M_{i+1}$, and $d(f_i(x_i), x_{i+1})<\delta$, we have
  \begin{equation*}
      \|D_\gamma\exp_{x_{i+1}}^{-1}\circ f_i\circ\exp_{x_i}-D_{x_i}f_i\|\le \eta, \forall\gamma\in T_{x_i}M_i(L\delta).
  \end{equation*}
   Denote by
  \begin{equation*}
      \exp_{x_{i+1}}^{-1}\circ f_i\circ\exp_{x_i}=(\beta_i^u, \beta_i^s), \forall i\in \mathbb{Z}.
  \end{equation*}
  Then
  \begin{equation*}
      F_{z_i}(w_i)=\beta_i^u(\Pi_{x_i}^sz_i+w_i)-\beta_i^u(\Pi_{x_i}^sz_i).
  \end{equation*}
  and, for any $\gamma\in T_{x_i}M_i(L\delta)$, we have
 \begin{align*}
      m(\frac{\partial}{\partial u}\beta^u_i(\gamma))&\ge m(B_{11i})-\|\frac{\partial}{\partial u}\beta_i^u(\gamma)-B_{11i}\|\\
      &\ge \lambda_u-\|\Pi_{x_{i+1}}^u\circ(D_\gamma\exp_{x_{i+1}}^{-1}\circ f_i\circ \exp_{x_i}-D_{x_i}f_i)\circ\Pi_{x_i}^u\|\\
      & \ge \lambda_u -h^2\eta \\
      &\ge \tilde{\lambda}_u.
  \end{align*}
For any $w_i, w'_i\in E_{x_i}^u(L\delta)$, let$\gamma(t)=\Pi_{x_i}^sz_i+w_i+t(w'_i-w_i), t\in[0, 1]$, 则$\gamma(t)\in T_{x_i}M_i(\rho)$.By the mean value theorem, 
\begin{align*}
    |F_{z_i}(w_i)-F_{z_i}(w'_i)|
    & =|\beta_i^u(\Pi_{x_i}^sz_i+w_i)-(\beta_i^u(\Pi_{x_i}^sz_i+w'_i)))|\\
    & = |\int_0^1\frac{\partial }{\partial t}(\beta_i^u(\gamma(t)))\, dt |\\
    & = |\int_0^1 \frac{\partial}{\partial u}\beta^u_i (\gamma(t))(w'_i-w_i)\, dt| \\
    & \geq \min_{t\in [0, 1]} m(\frac{\partial}{\partial u}\beta_i^u(\gamma(t)))|w'_i-w_i|\\
    & \geq \tilde{\lambda}_u|w_i-w'_i|.
\end{align*}
\end{proof}

\begin{lemma}\label{Q_vi property}
 $Q_{z_i}:=F_{z_i}^{-1}$is well-defined and continuous, and$|Q_{z_i}(w_i)|\leq\tilde{\lambda}_u^{-1}|w_i|.$
\end{lemma}

Let $k\in\mathbb{Z}^+, $ set $\mathbb{K}=\{-k, -k+1, \cdots, k-1, k\}$.For $i\in\mathbb{K}, $ set
\begin{align*}
   & \Phi_{\mathbb{K}}=\{\bm{z}=\{z_i\}_{i\in\mathbb{K}}:z_i\in T_{x_i}M_i\};\\
    &  \Phi_{\mathbb{K}}(L\delta) = \{\bm{z}\in\Phi_{\mathbb{K}}:\, |\Pi_{x_i}^sz_i|, |\Pi_{x_i}^uz_i|\le L\delta\};\\
&  \Phi_{\mathbb{K}, p}(L\delta) = \{\bm{z}\in\Phi_{\mathbb{K}}:\, \|\{\Pi_{x_i}^sz_i\}\|_p\le L\delta, \, \|\{\Pi_{x_i}^uz_i\}\|_p\le L\delta\}.
\end{align*}
It is obvious that $ \Phi_{\mathbb{K}, p}(L\delta) $is a closed subset of $ \Phi_{\mathbb{K}}(L\delta) $.
Introduce the operator$\mathcal{H}:\Phi_{\mathbb{K}}(L\delta)\to\Phi_{\mathbb{K}} , \bm{z}=\{z_i\}_{i\in\mathbb{K}}\mapsto\bm{ w}=\{w_i\}_{i\in\mathbb{K}}$ defined by
\begin{align}\label{stable}
    & \Pi_{x_{-k}}^sw_{-k} =0;\notag\\  &\Pi_{x_{i+1}}^sw_{i+1}=\Pi_{x_{i+1}}^s(G_i(z_i)), &i=-k, \cdots, k-1; \\
   &  \Pi_{x_k}^uw_k=0;\notag\\
  &   \Pi_{x_i}^uw_i=  Q_{z_i}(\Pi_{x_{i+1}}^u(-G_i(z_i)+F_i(z_i)-F_i(\Pi_{x_i}^sz_i)+z_{i+1})), &i=-k, \cdots, k-1.  \label{unstable} 
\end{align}
\begin{lemma}\label{H well defined}  
$\mathcal{H}:\Phi_{\mathbb{K}}(L\delta)\to\Phi_{\mathbb{K}}$ is well-defined.
\begin{proof} Take $\mu_u=\sup\limits_i \{\mu_{ui}\}, \tilde{\mu}_u>\mu_u$ and $0<\eta<\frac{\tilde{\mu}_u-\mu_u}{h^2}.$

Clearly, \eqref{stable} is well-defined for $\bm{z}\in\Phi_{\mathbb{K}}(L\delta)$.We only need to prove \eqref{unstable} is well-defined. By Lemma \ref{Q_vi property}, it is suffices to prove that, for $ i=-k, \cdots, k-1$, 
\begin{equation*}
    |\Pi_{x_{i+1}}^u(-G_i(z_i)+F_i(z_i)-F_i(\Pi_{x_i}^sz_i)+z_{i+1}|\le\tilde{\lambda}_uL\delta.
\end{equation*}
Denote
\begin{equation*}
       J_i=(\Pi_{x_{i+1}}^u(-G_i(z_i)+F_i(z_i)-F_i(\Pi_{x_i}^sz_i)+z_{i+1}).
\end{equation*}
let $J_i=J_{1i}+J_{2i}+J_{3i}$, where
\begin{align*}
    & J_{1i} = \Pi_{x_{i+1}}^u(-G_i(z_i)+F_i(z_i)+(0_{i+1}-F_i(0_i));\\
    & J_{2i} = \Pi_{x_{i+1}}^u(F_i(0_i)-F_i(\Pi_{x_i}^sz_i));\\
    & J_{3i} = \Pi_{x_{i+1}}^uz_{i+1}.
\end{align*}

To estimate $|J_{1i}|$, by definition, 
\begin{equation*}
    |G_i(z_i)-F_i(z_i)|\leq \sup_{x_i\in M_i} d(g_i(x_i), f_i(x_i))\le\delta.
\end{equation*}
On the other hand,  $F_i(0_i)=\exp_{x_{i+1}}^{-1}(f_i(x_i))$, By the definition of the pseudo-orbit  of the sequence $\{x_i\}_{i=k}^k$, 
\begin{equation*}
    |0_{i+1}-F_i(0_i)|\le \delta_{i+1}\le\delta.
\end{equation*}
So
\begin{equation}\label{J1}
    |J_{1i}|\leq 2h\delta. 
\end{equation}
Now estimate $|J_{2i}|$. For any $\gamma\in E_{x_i}^u(L\delta), $
\begin{align*}
  \|\frac{\partial}{\partial s}\beta_i^u(\gamma)\|&\le \|\frac{\partial}{\partial s}\beta_i^u(\gamma)-B_{12i}\|+\|B_{12i}\|\\
  &\le\|\Pi_{x_{i+1}}^u\circ(D_\gamma\circ\exp_{x_{i+1}}^{-1}\circ f_i\circ\exp_{x_i}-D_{x_i}f_i)\circ\Pi_{x_i}^s\|+\mu_u\\
  &\le h^2\eta+\mu_u\\
  &\le \tilde{\mu}_u.
\end{align*}
By the mean value theorem, 
\begin{equation*}
    \beta_i^u(0)-\beta_i^u(\Pi_i^sz_i)=\int_0^1D\beta_i^u(t\cdot\Pi_i^sz_i)dt\cdot(-\Pi_i^sz_i)
\end{equation*}
Thus, 

\begin{align}\label{J2}
    |J_{2i}|&=|\beta_i^u(0_i)-\beta_i^s(\Pi_i^sz_i)|\nonumber\\
    &\le \sup_{\gamma\in E_{x_i}^u(L\delta)}\|\frac{\partial}{\partial s}\beta_i^u(\gamma) \|\cdot|\Pi_{x_i}^sz_i|\\
    &\le \tilde{\mu}_sL\delta.\nonumber
\end{align}

By the definition of $z\in \Phi_{\mathbb{K}}(L\delta)$, it is obvous that
\begin{equation}\label{J3}
    |J_{3i}|=|\Pi_{x_{i+1}}^uz_{i+1}|\le L\delta.
\end{equation}
so, 
\begin{equation*}
    |J_i|\le(2h+\tilde{\mu}_uL+L)\delta
\end{equation*}
Take $L\ge\dfrac{2h}{\tilde{\lambda}_u-1-\tilde{\mu}_u}$. Then
\begin{equation}\label{J value}
    |J_i|\leq \tilde{\lambda}_uL\delta.
\end{equation}
\end{proof}
\end{lemma}

\begin{lemma}\label{H-self}
$\mathcal{H}(\Phi_{\mathbb{K}, p}(L\delta))\subset\Phi_{\mathbb{K}, p}(L\delta).$
\begin{proof}
  Let $\bm{z}\in\Phi_{\mathbb{K}}(L\delta), w=\mathcal{H}(z)$, denote:
  \begin{equation*}
      m^s(\bm{w})=\{\Pi_{x_i}^sw_i\}_{i\in\mathbb{K}}, \, \, m^u(\bm{w})=\{\Pi_{x_i}^uw_i\}_{i\in\mathbb{K}}.
  \end{equation*}
  (1) Unstable direction:
  By Lemma \ref{Q_vi property}, for each $i\in \mathbb{K}$, we have
  \begin{equation*}
      |\Pi_{x_i}^uw_i|\leq\tilde{\lambda}_u^{-1}(|J_{1i}|+|J_{2i}|+|J_{3i}|).
  \end{equation*}
  
 Let $J_1=\{J_{1i}\}_{i\in \mathbb{K}}, \, J_2=\{J_{2i}\}_{i\in \mathbb{K}}, \, J_3=\{J_{3i}\}_{i\in \mathbb{K}}, $ Taking the $P$-norm of \eqref{J1}, \eqref{J2}, and \eqref{J3}, we have
 \begin{align*}
     &\|J_1\|_p\le 2h\delta;\\
     &\|J_2\|_p\le\tilde{\mu}_u\|\{\Pi_{x_i}^sz_i\}\|_p\le\tilde{\mu}_uL\delta;\\
     &\|J_3\|_p\le\|\{\Pi_{x_{i+1}}^uz_{i+1}\}\|_p\le L\delta.
 \end{align*}
According to the properties of the p-norm and inequality \eqref{J value}:
\begin{align*}
        \|m^u(\mathcal{H}\bm{z})\|_p=\|m^u(\bm{w})\|_p
        &\leq\tilde{\lambda}_u^{-1}\left(\|J_1\|_p+\|J_2\|_p+\|J_3\|_p\right)\\
        & \leq \tilde{\lambda}_u^{-1}(2h\delta+\tilde{\mu}_uL\delta+L\delta)\\
        &\leq L\delta.
\end{align*}
(2) Stable direction: For equation \eqref{stable}, $i=-k, \cdots, k-1$, denote
\begin{equation*}
    \Pi_{x_{i+1}}^sw_{i+1}=I_{1i}+I_{2i}.
\end{equation*}
where, 
\begin{align*}
    &  I_{1i}=\Pi_{x_{i+1}}^s(G_i(z_i)-F_i(z_i)+F_i(0_i)-0_{i+1});\\
    &  I_{2i}=\Pi_{x_{i+1}}^s(F_i(z_i)-F_i(0_i)).
\end{align*}
Let $I_1=\{I_{1i}\}_{i\in\mathbb{K}}, \, I_2=\{I_{2i}\}_{i\in\mathbb{K}}$, \, then
\begin{equation*}
    \|m^s(\bm{w})\|_p\le \|I_1\|_p+\|I_2\|_p.
\end{equation*}
where, 
\begin{equation}\label{I1}
    |I_{1i}|\le h|G_i(z_i)-F_i(z_i)|+h|F_i(0_i)-0_{i+1}|, 
\end{equation}
and then
\begin{align*}
    \|I_1\|_p
    & \le h\|\bm{g}-\bm{f}\|_p+h\|\Delta\|_p\\
    & \le 2h\delta.
\end{align*}
On the other hand, for $I_{2i}$, by the continuity of $DF_i$ and the semi-hyperbolic property of $\bm{f}$, take $\lambda_s=\inf\limits_i\{\lambda^{-1}_{si}\}>1, \mu_s=\sup\limits_i\{\mu_{si}\}.$ For $ F_i=\exp_{x_{i+1}}^{-1}\circ f_i \circ \exp_{x_i}$, by lemma\ref{F_vi property}, $\exists1<\tilde{\lambda}_s<\lambda_s, \tilde{\mu}_u>\mu_u, $ such that $0<\eta<\min\{\frac{1-\tilde{\lambda}^{-1}_s}{h^2}, \frac{\tilde{\mu}_s-\mu_s}{h^2}\}, $ and
\begin{equation*}
    \|\frac{\partial}{\partial s}\beta^s_i(v)\|\le \tilde{\lambda}_s^{-1}\text{ and }\|\frac{\partial}{\partial u}\beta_i^s(v)\|\le \tilde{\mu}_s, \, \, \, \forall v\in T_{x_i}M_i(\rho).
\end{equation*}
So, 
\begin{align}\label{I2}
    |I_{2i}|
    & = |\beta^s_i(z_i)-\beta^s_i(0_i)|\notag\\
    & =|\int_0^1\frac{\partial}{\partial t }(\beta_i^s (tz_i))\, dt|\notag\\
    &=|\int_o^1(\frac{\partial}{\partial u}\beta^s_i(tz_i), \frac{\partial}{\partial s}\beta_i^u(tz_i))\cdot (\Pi_{x_i}^u z_i, \Pi _{x_i}^s z_i)^T dt| \notag\\
    & \leq \sup_{t\in[0, 1]}\|\frac{\partial}{\partial u}\beta_i^s(tz_i)\||\Pi_{x_i}^uz_i|+\sup_{t\in[0, 1]}\|\frac{\partial}{\partial s}\beta_i^s(tz_i))\||\Pi_{x_i}^sz_i|\notag\\
    &\leq \tilde{\mu}_s|\Pi_{x_i}^u z_i|+\tilde{\lambda}_s^{-1}|\Pi_{x_i}^s z_i|;
\end{align}
Furthermore, 
\begin{equation*}
        \|I_2\|_p\leq\tilde{\mu}_s\|\{\Pi_{x_i}^uz_i\}\|_p+\tilde{\lambda}_s^{-1}\|\{\Pi_{x_i}^sz_i\}\|_p\le(\tilde{\lambda}_s^{-1}+\tilde{\mu}_s)L\delta.
\end{equation*}
Therefore, when $L\ge\dfrac{2h}{1-\tilde{\lambda}_s^{-1}-\tilde{\mu}_s}$, we have
\begin{equation*}
    \|m^s(\bm{w})\|_p\le2h\delta+(\tilde{\lambda}_s^{-1}+\tilde{\mu}_s)L\delta\leq L\delta.
\end{equation*}
Then $w\in\Phi_{\mathbb{K}, p}(L\delta)$.
\end{proof}
\end{lemma}
\vspace{5pt}

\textit{Proof of Theorem}~\ref{Lp-finite}.
$\Phi_{\mathbb{K}, p}(L\delta)$ is a closed ball in a finite-dimensional space. From the proofs of Lemma \ref{F_vi property} and Lemma \ref{Q_vi property}, we know that the conclusion still holds for finite sequences. By Lemma \ref{H well defined}, $\mathcal{H}$ is continuous on the set $\Phi_{\mathbb{K}, p}(L\delta)$. By Lemma \ref{H-self}, $\mathcal{H}(\Phi_{\mathbb{K}, p}(L\delta))\subset\Phi_{\mathbb{K}, p}(L\delta)$. Therefore, by the Brouwer Fixed Point Theorem, there exists $\bm{z}=\{z_i\}_{i\in\mathbb{K}}\in\Phi_{\mathbb{K}, p}(L\delta)$ such that $\mathcal{H}(\bm{z})=\bm{z}$.

Now we show that this fixed point $\bm{z}$ is an orbit of $\bm{G}$, i.e., $z_{i+1}=G_i(z_i)$ for $i=-k, \cdots, k$ holds.
Since $\mathcal{H}(\bm{z})=\bm{z}$, equation\eqref{stable} can be rewritten as
\begin{equation}\label{sta}
   \Pi_{x_{i+1}}^sz_{i+1}=\Pi_{x_{i+1}}^s(G_i(z_i)).
\end{equation}
Similarly, equation\eqref{unstable} can be written as:
\begin{equation*}
       \Pi_{x_i}^uz_i=  Q_{z_i}(\Pi_{x_{i+1}}^u(-G_i(z_i)+F_i(z_i)-F_i(\Pi_{x_i}^sz_i)+z_{i+1})).
\end{equation*}
Applying $F_{z_i}=Q_{z_i}^{-1}$ to both sides of the last equation, we obtain
\begin{equation*}
    F_{z_i}(\Pi_{x_i}^uz_i)=\Pi_{x_{i+1}}^u(-G_i(z_i)+F_i(z_i)-F_i(\Pi_{x_i}^sz_i)+z_{i+1}).
\end{equation*}
Recall the definition of $F_{z_i}$:
\begin{equation*}
        F_{z_i}(\Pi_{x_i}^uz_i)=\Pi_{x_{i+1}}^u(F_i(\Pi_{x_i}^sz_i+\Pi_{x_i}^uz_i)-F_i(\Pi_{x_i}^sz_i)).
\end{equation*}
Comparing the last two equations, we have $\Pi_{x_{i+1}}^u(-G_i(z_i)+z_{i+1})=0$, then
\begin{equation}\label{unsta}
    \Pi_{x_{i+1}}^uz_{i+1}=\Pi_{x_{i+1}}^u(G_i(z_i)).
\end{equation}
In summary, from \eqref{sta} and \eqref{unsta}, we obtain
\begin{equation*}
    z_{i+1}=G_i(z_i), i=-k, \cdots, k-1.
\end{equation*}
Let $y_i=\exp_{x_i}(z_i)$, then 
\begin{equation*}
    y_{i+1}=\exp_{x_{i+1}}(z_{i+1})=\exp_{x_{i+1}}(G_i(z_i))=g_i(\exp_{x_i}(z_i))=g_i(y_i), 
\end{equation*}
and 
\begin{equation*}
    d(y_i, x_i)=d(\exp_{x_i}(z_i), x_i)=|z_i|\le|\Pi_{x_i}^sz_i|+|\Pi_{x_i}^uz_i|.
\end{equation*}
Then we have
\begin{equation*}
    \|\{d(x_i, y_i)\}_{i\in\mathbb{K}}\|_p\le\|\{\Pi_{x_i}^sz_i\}\|_p+\|\{\Pi_{x_i}^uz_i\}\|_p\le2L\delta.
\end{equation*}
Take
\begin{equation*}
    \tilde{L}=\max\left\{\dfrac{2h}{\tilde{\lambda}_u-1-\tilde{\mu}_u}, \dfrac{2h}{1-\tilde{\lambda}_s^{-1}-\tilde{\mu}_s}\right\}, 
\end{equation*}
then
\begin{equation*}
    \|\{d(y_i, x_i)\}_{i\in\mathbb{K}}\|_p\le \tilde{L}\|\Delta\|_p.
\end{equation*}
Therefore, we have proved that there exist constants $\tilde{L}, \delta>0$ such that for any finite $\Delta$-pseudo-orbit $\{x_i\}_{i=-k}^k$  with $\|\Delta\|_p\le \delta$, and any $\bm{g}$ with $\|\bm{f}-\bm{g}\|_p\le\delta$, there exists an orbit $\{y_i\}_{i=-k}^k$ of $\bm{g}$ such that
\begin{equation*} 
    \|\{d(x_i, y_i)\}_{i\in\mathbb{K}}\|_p\le L\|\Delta\|_p. 
\end{equation*}
This completes the proof. \hfill{$\square$}

\subsection{Proof of the \texorpdfstring{$L^p$}--bi-shadowing property}

\textit{Proof of Theorem \ref{Lp-bi}}. Given a pseudo-orbit $\bm{x}=\{x_i\}_{i\in\mathbb{Z}}$, for each $k\in\mathbb{N}$, define the finite sequence
    \begin{equation*}
        \bm{x}^{(k)}=\left\{ x_{-k}^{(k)}, \cdots, x_0^{(k)}, \cdots, x_k^{(k)} \right\}, \, \, \, k=1, 2, \cdots
    \end{equation*}
    where $x_i^{(k)}=x_i$ for $i=-k, \cdots, k$. It is obvious that  $\bm{x}^{(k)}$ is a finite $\Delta^{(k)}$-pseudo-orbit of $\bm{f}$, with $\Delta^{(k)}=\{\delta_i\}_{i=-k}^k$, and $\|\Delta^{(k)}\|_p\le\|\Delta\|_p\le\delta.$
    
    For the continuous map $\bm{g}:\mathbf{M}\to\mathbf{M}$ with $\|\bm{f}-\bm{g}\|_p\le\delta$, by Theorem \ref{Lp-finite}, there exist constants $L, \delta>0$ and a finite orbit $\bm{y}^{(k)}=\left\{ y_{-k}^{(k)}, \cdots, y_0^{(k)}, \cdots, y_k^{(k)} \right\}$ of $\bm{g}$ satisfying
    
    (1) $y_{i+1}^{(k)}=g_i(y_i^{(k)})$ for $i=-k, \cdots, k-1$;
    
    (2) Denote $z_i^{(k)}=\exp_{x_i}^{-1}(y_i^{(k)})$, then
\begin{align}\label{序列}
&|z_i^{(k)}|=d(x_i, y_i^{(k)})=d(x_i^{(k)}, y_i^{(k)});\\
 &\|\{|z_i^{(k)}|\}_{i=-k}^k\|_p\le L\|\Delta\|_p
\end{align}
for $i=-k, \cdots, k-1, k$.

In particular, for each $i$ we have the pointwise estimate, 
\begin{equation}\label{d(x_i, y_i)-value}
d(x_i, y_i^{(k)})=|z_i^{(k)}|\le L\|\Delta\|_p, \, \, \, \, \, (|i|\le k).
\end{equation}

Fix $i$. Consider the sequence $\{y_i^{(k)}\}_{|i|\le k}$. By \eqref{d(x_i, y_i)-value}, this sequence lies in the closed ball in the compact manifold $M_i$ centered at $x_i$ with radius $L\|\Delta\|_p$. Thus, the sequence $\{y_i^{(k)}\}_{|i|\le k}$ has a convergent subsequence.

Let $I_0=\mathbb{N}=\{1, 2, 3, \cdots\}$. For $i=0$, there exists $I_1\subset I_0$ an increasing sequence such that $\{y_0^{(k)}\}_{k\in I_1}$ converges; denote its limit by $y_0$.

For $i=1$, for the sequence $\{y_1^{(k)}\}_{k\in I_1, k\ge 1}$, there exists $I_2\subset I_1$ such that $\{y_1^{(k)}\}_{k\in I_2}$ converges; denote its limit by $y_1$.

For $i=-1$, for the sequence $\{y_{-1}^{(k)}\}_{k\in I_2, k\ge 1}$, there exists $I_3\subset I_2$ such that $\{y_{-1}^{(k)}\}_{k\in I_3}$ converges; denote its limit by $y_{-1}$.

Continue this process for $i=2, i=-2, \cdots$.

Generally, assume after the $m$-th step, we have nested index sets
\begin{equation*}
    I_0\supset I_1\supset I_2\supset\cdots\supset I_m
\end{equation*}
and limit points $y_0, y_1, y_{-1}, \cdots, y_{i_m}$ (where $i_m$ is the $m$-th integer in the order described). For the sequence $\{y_{i_{m+1}}^{(k)}\}_{k\in I_{m+1}, k\ge |i_{m+1}|}$, this sequence is still in the compact manifold $M_{i_{m+1}}$, so there exists $I_{m+1}\subset I_m$ such that $\{y_{i_{m+1}}^{(k)}\}_{k\in I_{m+1}}$ converges; denote its limit by $y_{i_{m+1}}$.
Thus, we obtain an infinite sequence of nested index sets:
\begin{equation*}
    I_0\supset I_1\supset I_2\supset\cdots\supset I_m\supset\cdots
\end{equation*}
Define the diagonal indices:
\begin{equation*}
    l_m:= \text{the $m$-th smallest element in } I_m, \, \, \, m=1, 2, 3, \cdots
\end{equation*}

By construction: $l_1<l_2<l_3<\cdots$ and $l_m\in I_m$.
For any integer $i$, suppose $i$ is the $m_i$-th integer processed in the order. Then for $m\ge m_i$, we have $l_m\in I_{m_i}$, and $\{y_i^{(l_m)}\}_{m\ge m_i}$ is a subsequence of the convergent sequence $\{y_i^{(k)}\}_{k\in I_{m_i}}$. Hence, the limit
\begin{equation*}
    y_i:=\lim_{m\to\infty}y_i^{(l_m)}
\end{equation*}
exists. Let $\bm{y}=\{y_i\}_{i\in\mathbb{Z}}$.
For any integer $i$, take sufficiently large $m$ such that $l_m\ge|i|+1$. By the finite orbit property, 
\begin{equation*}
    y_{i+1}^{(l_m)}=g_i(y_i^{(l_m)})
\end{equation*}
Letting $m\to \infty$, by the continuity of $g_i$, we obtain
\begin{equation*}
    y_{i+1}=g_i(y_i).\, \, \, \, \forall i\in\mathbb{Z}.
\end{equation*}
Therefore, $\bm{y}$ is an orbit of $\bm{g}$.

Let $z_i=\exp_{x_i}^{-1}(y_i)$. By the continuity of the exponential map, 
\begin{equation*}
    z_i=\lim_{m\to\infty} z_i^{(l_m)}
\end{equation*}
Extend each finite sequence $z^{(l_m)}$ by zero to an infinite sequence:
\begin{equation*}
    \tilde{z}_i^{(l_m)}=
    \begin{cases}
        z_i^{(l_m)}, &|i|\le l_m;\\
        0, & |i|>l_m.
    \end{cases}
\end{equation*}
From \eqref{序列}, we obtain
\begin{equation*}\label{zi估计}
    \|\tilde{z}^{(l_m)}\|_p^p=\sum_{|i|\le l_m}|z_i^{(l_m)}|^p\le(L\|\Delta\|_p)^p.
\end{equation*}
For each fixed $i$, as $m\to \infty$, we have pointwise convergence $\tilde{z}_i^{(l_m)}\to z_i$, and since $|\tilde{z}_i^{(l_m)}|^p\ge 0$, applying Fatou's Lemma:
\begin{equation*}
    \sum_{i\in \mathbb{Z}}|z_i|^p=\sum_{i\in\mathbb{Z}}\liminf_{m\to\infty}|\tilde{z}_i^{(l_m)}|^p\le\liminf_{m\to\infty}\sum_{i\in\mathbb{Z}}|\tilde{z}_i^{(l_m)}|^p.
\end{equation*}
Then by \eqref{zi估计}, 

\begin{equation*}
\sum_{i\in\mathbb{Z}}|z_i|^p\le(L\|\Delta\|_p)^p
\end{equation*}
Since the limit exists, $\liminf\limits_{k\to\infty}|z_n^{(k)}|^p=\lim\limits_{k\to\infty}|z_n^{(k)}|^p$, therefore
\begin{equation*}
    \|\{d(x_i, y_i)\}\|_p=\|\{|z_i|\}\|_p\le L\|\Delta\|_p.  
\end{equation*}
This completes the proof. \hfill{$\square$} 

\section{Limit bi-shadowing property}

In this section, we study the limit bi-shadowing property of semi-hyperbolic systems. The limit bi-shadowing property requires that the system not only has the usual shadowing property but also that the shadowing error tends to zero when the pseudo-orbit error tends to zero. This property is important in stability analysis and perturbation theory.

\begin{definition}\label{limit-bi-def}
    Let $\bm{f}:\mathbf{M}\to \mathbf{M}$, where $\mathbf{M}=\bigsqcup\limits_{i\in\mathbb{Z}}M_i$. For any sequence $\Delta=\{\delta_i\}_{i\in\mathbb{Z}}$, let $\bm{x}=\{x_i\}_{i\in\mathbb{Z}}$ be a $\Delta$-pseudo-orbit of $\bm{f}$ (i.e., $d(f_i(x_i), x_{i+1})\le\delta_i$ for all $i\in\mathbb{Z}$). Let $\bm{g}:\mathbf{M}\to \mathbf{M}$ be a family of continuous maps. Define
    \begin{equation*}
        \|\bm{f}-\bm{g}\|_{\infty}:=\sup_{i\in\mathbb{Z}}\sup_{x_i\in M_i}d(f_i(x_i), g_i(x_i)).
    \end{equation*}
    If $\bm{f}$ satisfies the following conditions:
    \begin{enumerate}[label=\upshape(\arabic*)]
      \item $\bm{f}$ has the $L^\infty$-bi-shadowing property, i.e., there exist constants $ L, \delta>0, $ such that for any sequence $\Delta$ with $\|\Delta\|_\infty\le \delta$, and any family of continuous maps $\bm{g}$ with $\|\bm{f}-\bm{g}\|_\infty\le\delta, $ if $\bm{x}$ is a $\Delta$-pseudo-orbit of $\bm{f}$, then there exists an orbit $\bm{y}=\{y_i\}_{i\in
      \mathbb{Z}}$ of $\bm{g}$, satisfying $y_{i+1}=g_i(y_i)$ for all $i\in\mathbb{Z}$, such that
      \begin{equation*}
          \|\{d(x_i, y_i)\}_{i\in\mathbb{Z}}\|_\infty\le L\|\Delta\|_\infty;
      \end{equation*}

      \item When $\lim\limits_{|i|\to\infty} \delta_i = 0$ and $\lim\limits_{|i|\to \infty} d(f_i(x_i), g_i(x_i))=0$, the corresponding orbit $\bm{y}$ satisfies
\begin{equation*}
    \lim\limits_{|i|\to\infty}d(x_i, y_i)=0.
\end{equation*}
  \end{enumerate}
   Then $\bm{f}$ is said to have the limit bi-shadowing property.
\end{definition}

\textit{Proof of Theorem \ref{limit-bi}}.
Take  $L>\max\left\{\dfrac{3h}{\tilde{\lambda}_u-2-\tilde{\mu}_u}, \dfrac{3h}{1-2(\tilde{\lambda}_s^{-1}-\tilde{\mu}_s)}\right\}$. Now assume $\eta_i=d(f_i(x_i), g_i(x_i))\to0$ and $\lim\limits_{i\to\infty}\delta_i=0$. Let $\varepsilon_0= L\delta$, define $\varepsilon_n=\varepsilon_{n-1}/2$. Since $\delta_i\to 0$, for each $n\in\mathbb{N}, \, \varepsilon_n>0$, there exists $ k_n\in\mathbb{N}$ such that for $|i|\ge k_n$, we have $\delta_i<\varepsilon_n, \eta_i<\varepsilon_n$.

  Given $l\in\mathbb{Z}^+$, define the sets
    \begin{align*}
        &\tilde{\Phi}=\{\bm{z}=\{z_i\}_{i=-k_l}^{k_l}:z_i\in  T_{x_i}M_i(L\delta) \}\\
        &\tilde{\Phi}_n = \{\bm{z}\in \tilde{\Phi}:|z_i|\le L\varepsilon_n, |i|\in[k_n, k_{n+1}), n=0, 1, \cdots, l\}.
    \end{align*}
Define the operator: $\mathcal{H}_n:\tilde{\Phi}_n\to \Phi$, mapping $\bm{z}\in \tilde{\Phi}_n$ to $\bm{w}=\mathcal{H}_n(\bm{z})$:\begin{align}
   & \Pi_{x_{-k_l}}^sw_{-k_l}=0, \notag\\
    &\Pi_{x_{i+1}}^sw_{i+1}=\Pi_{x_{i+1}}^sG_i(z_i), &i=-k_l, \cdots, k_l-1;\\
    &\Pi_{x_{k_l}}^uw_{k_l}=0, \notag\\
    &\Pi_{x_i}^uw_i=Q_{z_i}(\Pi_{x_{i+1}}^u(-G_i(z_i)+F_i(z_i)-F_i(\Pi_{x_i}^sz_i)+z_{i+1}), &i=-k_l, \cdots, k_l-1.
\end{align}
  
Given $n=0, 1, \cdots, l-1$.  Now we prove $\mathcal{H}_n(\tilde\Phi_n)\subset\tilde\Phi_n.$

(1) Unstable direction, :
To estimate $J_{1i}, J_{2i}, J_{3i}:$

If $i\in[k_n, k_{n+1})$, by equation \eqref{J1}, \eqref{J2}, and \eqref{J3}, 
\begin{align*}
    &|J_{1i}|\le |G_i(z_i)-F_i(z_i)|+|0_{i+1}-F_i(0)|\le h \eta_i+h\delta_{i+1}\le 2h\varepsilon_n;\\
    &|J_{2i}|\le\tilde{\mu}_u|\Pi_{x_i}^sz_i|\le\tilde{\mu}_uL\varepsilon_n;\\
    &|J_{3i}|=|\Pi_{x_{i+1}}^uz_{i+1}|\le L\varepsilon_{n}.
\end{align*}
we obtain, 
\begin{equation*}
    |J_i|\le2h \varepsilon_n+(\tilde{\mu}_u+1)L\varepsilon_n, 
\end{equation*}
\begin{equation*}
    |\Pi_{x_i}^sw_i|\le\tilde{\lambda}_u^{-1}(2h\varepsilon_n+(\tilde{\mu}_u+1)L\varepsilon_n)
\end{equation*}

If$i\in(-k_{n+1}, -k_n], $有$\delta_i<\varepsilon_n, \delta_{i+1}\le \varepsilon_{n-1}$, By the definition of $\varepsilon_n$, then
\begin{align*}
    |\Pi_{x_i}^uw_i|&\le\tilde{\lambda}_u^{-1}(h\eta_i+h\delta_{i+1}+\tilde{\mu}_u|\Pi_{x_i}^sz_i|+|\Pi_{x_{i+1}}^uz_{i+1}|)\\
    &\le\tilde{\lambda}_u^{-1}(h\varepsilon_n+h\varepsilon_{n-1}+\tilde{\mu}_uL\varepsilon_n+L\varepsilon_{n-1})\\
    & = \tilde{\lambda}_u^{-1}(h\varepsilon_n+2h\varepsilon_n+\tilde{\mu}_uL\varepsilon_n+2L\varepsilon_n)\\
    &\le\tilde{\lambda}_u^{-1}[3h+(\tilde{\mu}_u+2)L\varepsilon_n]
\end{align*}
Since $L>\dfrac{3h}{\tilde{\lambda}_u-2-\tilde{\mu}_u}$, for $|i|\in[k_n, k_{n+1})$, we have
\begin{equation*}
    |\Pi_{x_i}^uw_i|\le L\varepsilon_n
\end{equation*} holds.

(2) Stable direction:

If $i\in[k_n, k_{n+1})$, by equation \eqref{I1} and \eqref{I2}, 
\begin{align*}
     &|I_{1i}|\le h\delta_i+h\delta_{i+1};\\
     &|I_{2i}|\le\tilde{\lambda }_s|\Pi_{x_i}^sz_i|+\tilde{\mu}_s|\Pi_{x_i}^uz_i|.
 \end{align*}
 Since $\Pi_{x_i}^sw_i=I_{1(i-1)}+I_{2(i-1)}$, we get
 \begin{align*}
     |\Pi_{x_{i}}^sw_i|&\le h\varepsilon_{n-1}+h\varepsilon_n+\tilde{\lambda}_sL\varepsilon_{n-1}+\tilde{\mu}_sL\varepsilon_{n-1}\\
     &\le[2h+(\tilde{\lambda }_s+\tilde{\mu}_s)L]\varepsilon_n\le L\varepsilon_n.
 \end{align*}
 
 If  $i\in(-k_n, -k_{n+1}]$, similarly, $\delta_i<\varepsilon_n, \delta_{i+1}\le \varepsilon_{n-1}$, 
\begin{align*}
     |\Pi_{x_i}^sw_i|&\le h\cdot2\varepsilon_n+h\varepsilon_n+\tilde{\lambda}_s^{-1}L\cdot2\varepsilon_n+\tilde{\mu}_sL\cdot2\varepsilon_n\\
     &\le[3h+(\tilde{\lambda }_s^{-1}+\tilde{\mu}_s)2L]\varepsilon_n\\
     &\le L\varepsilon_n.
 \end{align*} 

 Therefore, for each $n$, $\tilde{\Phi}_n$ is a closed convex set, and $\mathcal{H}_n$ maps $\tilde{\Phi}_n$ into itself. By the Brouwer Fixed Point Theorem, $\mathcal{H}_n$ has a fixed point $\bm{z}^{(n)}\in\tilde{\Phi}_n$ satisfying $z_{i+1}^{(n)}=G_i(z_i^{(n)})$ for $i\in\mathbb{Z}$. Let $y_i^{(n)}=\exp_{x_i}(z_i^{(n)})$, then $y_{i+1}^{(n)}=g_i(z_i^{(n)})$, and clearly
    \begin{equation*}
        d(x_i, y_i^{(n)})= d(x_i, \exp_{x_i}(z_i^{(n)}))=|z_i^{(n)}|.
    \end{equation*}
    And $|z_i^{(n)}|\le L\varepsilon_n$ holds for $|i|\ge k_n$.

   For each fixed $i$, as $n\to\infty$, $k_n\to\infty$, so for any, since $\{y_i^{(n)}\}$ is uniformly bounded on the compact manifold, there exists a convergent subsequence. By a method similar to Theorem \ref{Lp-bi}, we can obtain an index sequence $\{n_m\}$ such that for each $i$, $\{y_i^{n_m}\}$ converges. Denote
    \begin{equation*}
        y_i:=\lim_{m\to\infty}y_i^{(n_m)}.
    \end{equation*}
By the continuity of $\bm{g}$, $\{y_i\}$ is an orbit of $\bm{g}$.

Since $\varepsilon_n\to0$, for any $\eta>0$, there exists $N$ such that for $m\ge N$, $L\varepsilon_{n_m}<\frac{\eta}{2}$. Fix $m_0$, then for any $|i|\ge k_{n_{m_0}}$, we have
    \begin{equation*}
        d(x_i, y_i^{(n_{m_0})})\le L\varepsilon_{n_{m_0}}<\frac{\eta}{2}.
    \end{equation*}
    By pointwise convergence, there exists $m_1\ge m_0$ such that
   \begin{equation*}
       d(y_i^{(n_{m_1})}, y_i)<\frac{\eta}{2}. \end{equation*}
        Since $n_{m_1}\ge n_{m_0}$, we have $k_{(n_{m_1})}\ge k_{n_{m_0}}=N'$. Therefore, for $|i|\ge N'$, also satisfying $|i|\ge k_{n_{m_1}}$, by the decreasing property of $\varepsilon_n$, we have $d(x_i, y_i^{(n_{m_1})})\le L\varepsilon_{n_{m_1}}\le L\varepsilon_{n_{m_0}}<\frac{\eta}{2}$.
By the triangle inequality, 
\begin{equation*}
    d(x_i, y_i)\le d(x_i, y_i^{(n_{m_1})})+d(y_i(n_{m_1}), y_i)<\frac{\eta}{2}+\frac{\eta}{2}=\eta.
\end{equation*}
Thus, for any $\eta>0$, there exists $N'=k_{n_{m_0}}$ such that for $|i|\ge N'$, $d(x_i, y_i)<\eta$, i.e., 
\begin{equation*}
    \lim\limits_{|i|\to\infty} d(x_i, y_i)=0.\tag*{$\square$} 
\end{equation*}

\section{Asymptotic bi-shadowing property}

This section further studies the asymptotic bi-shadowing property, which examines whether, when the pseudo-orbit error and the map perturbation tend to zero exponentially, the system can still be approximated by the orbit of a nearby system at the corresponding exponential rate. This property plays an important role in studying the stability and orbit approximation behavior of systems in an asymptotic sense.

\textit{Proof of Theorem \ref{Asy-bi}} Let $\{x_i\}_{i\in\mathbb{Z}}$ be a $\Delta$-pseudo-orbit of $\bm{f}$, where $\Delta=\{\delta_i\}_{i\in\mathbb{Z}}$  with
    \begin{equation*}
        \limsup_{|i|\to\infty}\sqrt[|i|]{\delta_i}=v<1, 
    \end{equation*}
then any $\varepsilon\in(0, 1-v)$, there exists $k_0\in \mathbb{N}$ such that for $|i|\ge k_0$, $\sqrt[|i|]{\delta_i}<v+\varepsilon$, i.e., 
    \begin{equation*}
        \delta_i<(v+\varepsilon)^{|i|}, \forall|i|\ge k_0.
    \end{equation*}
   Since $\varepsilon\in(0, 1-v)$, we have $v+\varepsilon<1$. Take $k_1>k_0$. When $k_1-k_0$ is sufficiently large, we can have $(v+\varepsilon)^{k_1-k_0}\le \delta$, where $\delta>0$ is the constant in Theorem \ref{Lp-finite} that guarantees the finite $L^p$-bi-shadowing property.
     Let the family of continuous maps $\bm{g}:\mathbf{M}\to\mathbf{M}$ with $\sup\limits_{x_i\in M_i}d(f_i(x_i), g_i(x_i))\le\delta_i$, and $\limsup\limits_{|i|\to\infty}\sqrt[|i|]{\delta_i}\le v.$

    Take $L>\max\left\{\dfrac{h+h(v+\varepsilon)^{-1}}{\tilde{\lambda}_u-\tilde{\mu}_u-(v+\varepsilon)^{-1}}, \dfrac{h+h(v+\varepsilon)^{-1}}{1-(v+\varepsilon)^{-1}(\tilde{\lambda}_s^{-1}-\tilde{\mu}_s)}\right\}$.Set
    \begin{equation*}
        \hat{\Phi}_k=\{\bm{z}\in \tilde{\Phi} :|z_i|\le L\delta, |i|\le k_1;|z_i|\le L\delta(v+\varepsilon)^{|i|-k_1}, |i|\ge k_1 \}.
    \end{equation*}

    Define the operator $\mathcal{H}:\hat{\Phi}_k\to \tilde{\Phi}$ following the definition in Theorem \ref{Lp-finite} which transforms $\bm{z}=\{z_i\}_{i=-k}^k$ into a sequence $\bm{w}=\{w_i\}_{i=-k}^k$. Next, we prove $\mathcal{H}:\hat{\Phi}_k\subset\hat{\Phi}_k$.
    
    (1) Unstable direction:
    
    If $i=k$, by definition $|\Pi_{x_i}^uw_i|=0;$
    
   If $k_1\le i\le k-1$, by equation \eqref{J1}, \eqref{J2}, and \eqref{J3}, 
    \begin{align*}
       & |J_1|\le h\delta_i+h\delta_{i+1}\le h(v+\varepsilon)^{i}+h(v+\varepsilon)^{i+1}\le2h(v+\varepsilon)^i;\\
        &|J_2|\le \tilde{\mu}_u|\Pi_{x_i}^sz_i|\le \tilde{\mu}_uL\delta(v+\varepsilon)^{i-k_1};\\
        &|J_3|\le |\Pi_{x_{i+1}}^uz_{i+1}|\le L\delta(v+\varepsilon)^{i+1-k_1}\le L\delta(v+\varepsilon)^{i-k_1}.
    \end{align*} 
    Since $(v+\varepsilon)^{k_1-k_0}\le \delta$, then $(v+\varepsilon)^{k_1}\le(v+\varepsilon)^{k_0}\delta\le\delta$. Therefore, 
    \begin{align*}
        |\Pi_{x_i}^uw_i|&\le\tilde{\lambda}^{-1}_u[ 2h(v+\varepsilon)^{i}+\tilde{\mu}_uL\delta(v+\varepsilon)^{i-k_1}+L\delta(v+\varepsilon)^{i-k_1}]\\
        &\le\tilde{\lambda}_u^{-1}(v+\varepsilon)^{i-k_1}[2h(v+\varepsilon)^{k_1}+(\tilde{\mu}_u+1)L\delta]\\
        &\le \tilde{\lambda}_u^{-1}(v+\varepsilon)^{i-k_1}[2h\delta+(\tilde{\mu}_u+1)L\delta]\\
        &\le  L\delta(v+\varepsilon)^{i-k_1};
    \end{align*}
    
    If $-k_1\le i\le k_1-1$, by Lemma \ref{H-self}, for $p=\infty$, 
    \begin{equation*}
        |\Pi_{x_i}^uw_i|\le L\delta, \forall i = -k_1, \cdots, k_1-1.
    \end{equation*}
    
    If $-k\le i\le-k_1$, $|i+1|=-i-1=|i|-1$, we have
    \begin{align*}
        |\Pi_{x_i}^uw_i|&\le\tilde{\lambda}_u^{-1}(h\delta_i+h\delta_{i+1}+\tilde{\mu}_u|\Pi_{x_i}^sz_i|+|\Pi_{x_i}^uz_{i+1}|)\\
        &\le\tilde{\lambda}_u^{-1}[h(v+\varepsilon)^{|i|}+h(v+\varepsilon)^{|i|-1}+\tilde{\mu}_uL\delta(v+\varepsilon)^{|i|-k_1}+L\delta(v+\varepsilon)^{|i|-1-k_1}]\\
        &\le\tilde{\lambda}_u^{-1}(v+\varepsilon)^{|i|-k_1}[h(v+\varepsilon)^{k_1}(1+(v+\varepsilon)^{-1})+(\tilde{\mu}_u+(v+\varepsilon)^{-1})L\delta]\\
        &\le \tilde{\lambda}_u^{-1}(v+\varepsilon)^{|i|-k_1}[h\delta(1+(v+\varepsilon)^{-1})+(\tilde{\mu}_u+(v+\varepsilon)^{-1})L\delta]\\
      &\le  L\delta(v+\varepsilon)^{|i|-k_1}.
    \end{align*}

  (2) Stable direction:
    
    If $k_1\le i\le k$, by equation \eqref{I1} and \eqref{I2}, 
    
    \begin{align*}
       & |I_{1i}|\le h\delta_i+h\delta_{i+1}\le h(v+\varepsilon)^{i}+h(v+\varepsilon)^{i+1};\\
       &|I_{2i}|\le \tilde{\lambda}_s^{-1}|\Pi_{x_i}^sz_i|+ \tilde{\mu}_s|\Pi_{x_i}^uz_i| \le(\tilde{\lambda}_s^{-1}+\tilde{\mu}_s)L\delta(v+\varepsilon)^{i-k_1}.
    \end{align*}
   Since $\Pi_{x_i}^sw_i=I_{1(i-1)}+I_{2(i-1)}$, we have
    \begin{align*}
        |\Pi_{x_i}^sw_i|&\le h(v+\varepsilon)^{i-1}+h(v+\varepsilon)^{i}+(\tilde{\lambda}_s^{-1}+\tilde{\mu}_s)L\delta(v+\varepsilon)^{i-1-k_1}\\
        &\le (v+\varepsilon)^{i-k_1}[h(v+\varepsilon)^{k_1}((v+\varepsilon)^{-1}+1)+(\tilde{\lambda}_s^{-1}+\tilde{\mu}_s)(v+\varepsilon)^{-1}L\delta]\\
        & \le(v+\varepsilon)^{i-k_1}[h\delta((v+\varepsilon)^{-1}+1)+(\tilde{\lambda}_s^{-1}+\tilde{\mu}_s)(v+\varepsilon)^{-1}L\delta]\\
        & \le  L\delta(v+\varepsilon)^{i-k_1}.
    \end{align*}

If $-k_1< i< k_1$by Lemma \ref{H-self}, taking $p=\infty$, we get
\begin{equation*}
    |\Pi_{x_i}^sw_i|\le L\delta, \forall i = -k_1+1, \cdots, k_1-1.
\end{equation*}

If $-k+1\le i \le -k_1$, similarly, $|i-1|=|i|+1$, 
\begin{align*}
    |\Pi_{x_i}^sw_i|&\le h\delta_{i-1}+h\delta_{i}+ \tilde{\lambda}_s^{-1}|\Pi_{x_{i-1}}^sz_{i-1}|+ \tilde{\mu}_s|\Pi_{x_{i-1}}^uz_{i-1}| \\
    &\le h(v+\varepsilon)^{|i-1|}+h(v+\varepsilon)^{|i|}+(\tilde{\lambda}_s^{-1}+\tilde{\mu}_s)L\delta(v+\varepsilon)^{|i-1|-k_1}\\
      &\le h(v+\varepsilon)^{|i|+1}+h(v+\varepsilon)^{|i|}+(\tilde{\lambda}_s^{-1}+\tilde{\mu}_s)L\delta(v+\varepsilon)^{|i|+1-k_1}\\
    &\le (v+\varepsilon)^{|i|-k_1}[h(v+\varepsilon)^{k_1}((v+\varepsilon)+1)+(\tilde{\lambda}_s^{-1}+\tilde{\mu}_s)(v+\varepsilon)L\delta]\\
    &\le (v+\varepsilon)^{|i|-k_1}[2h\delta+(\tilde{\lambda}_s^{-1}+\tilde{\mu}_s)L\delta]\\
    &\le(v+\varepsilon)^{|i|-k_1};
\end{align*}

If $i=-k$, $|\Pi_{x_i}^sw_i|=0.$

So, \[|\Pi_{x_i}^uz_i^{(k)}|, |\Pi_{x_i}^sz_i^{(k)}|\le
\begin{cases}
    L\delta, &|i|\le k_1;\\
    (v+\varepsilon)^{|i|-k_1}, & |i|>k_1.
\end{cases}
\]
Thus, $\mathcal{H}$ maps $\hat{\Phi}_n$ into itself. By the Brouwer Fixed Point Theorem, $\mathcal{H}$ has a fixed point $\bm{z}^{(k)}\in\hat{\Phi}_n$ satisfying $z_{i+1}^{(k)}=G_i(z_i)$. Let $y_{i+1}=\exp_{x_i}(z_i^{(k)})$, then $y_{i+1}^{(k)}=g_i(y_i^{(k)})$. By Theorem \ref{Lp-bi}, for any $i\in\mathbb{Z}$, we can assume $\lim\limits_{k\to\infty}z_{i+1}^{(k)}=z_i$, then
\[|\Pi_{x_i}^uz_i|, |\Pi_{x_i}^sz_i|\le
\begin{cases}
    L\delta, &|i|\le k_1;\\
    (v+\varepsilon)^{|i|-k_1}, & |i|>k_1.
\end{cases}
\]
Thus $\lim\limits_{k\to\infty}y_{i+1}^{(k)}=y_i$ and $\bm{y}=\{y_i\}_{i\in\mathbb{Z}}$ is an orbit of $\bm{g}$, i.e., $y_{i+1}=g_i(y_i).$

For $|i|\ge k_1$, we have
\begin{equation*}
        d(x_i, y_i)=d(\exp_{x_i}(z_i), x_i)=|z_i|\le 2L\delta(v+\varepsilon)^{|i|-k_1}
    \end{equation*}
   Obviously, 
    \begin{align*}
        \sqrt[|i|]{d(x_k, y_k)}&\le [2L\delta(v+\varepsilon)^{|i|-k_1}]^{1/{|i|}}\\
        &\le (2L\delta)^{1/|i|}(v+\varepsilon)^{1-k_1/|i|}
    \end{align*}
   With $|i|\to \infty$, $   (2L\delta)^{1/|i|}\to 1,  (v+\varepsilon)^{1-k_1/|i|}\to v+\varepsilon$.By the arbitrariness of $\varepsilon$, we obtain
    \begin{equation*}
        \limsup_{|i|\to\infty}\sqrt[|i|]{d(x_i, y_i)}\le v.
    \end{equation*}
    This completes the proof. \hfill{$\square$}

\newpage
\addcontentsline{toc}{section}{References}
%\bibliography{ref1}

\end{document}